\documentclass[11pt,reqno]{amsart}

\usepackage{amsmath}
\usepackage[left=1in,right=1in,top=1in,bottom=1in]{geometry}
\usepackage{url}
\usepackage{hyperref}

\newcommand{\cln}[2]{ \left\lceil \frac{#1}{#2} \right\rceil}

\newcommand{\SEQ}[4]{{\langle #1 ; #2 : #3 ; #4 \rangle}}

\theoremstyle{plain}
\newtheorem{theorem}{Theorem}[section]
\newtheorem{lemma}[theorem]{Lemma}

\numberwithin{equation}{section}
\numberwithin{theorem}{section}
\numberwithin{table}{section}
\numberwithin{figure}{section}

\begin{document}


\title{Sums of ceiling functions solve nested recursions}

\author{Rafal Drabek \and Abraham Isgur \and Vitaly Kuznetsov \and Stephen M. Tanny}

\address{Department of Mathematics\\
University of Toronto\\
40 St. George Street\\
Toronto\\
ON M5S 2E4\\
Canada}

\email[Rafal Drabek]{rafal.drabek@utoronto.ca}
\email[Abraham Isgur]{umarovi@gmail.com}
\email[Vitaly Kuznetsov]{v.kuznetsov@utoronto.ca}
\email[Steve Tanny]{tanny@math.utoronto.ca}



\maketitle

\begin{abstract}
It is known that, for given integers $s \geq 0$ and $j>0$, the nested recursion $R(n) = R(n-s-R(n-j))+R(n-2j-s-R(n-3j))$ has a closed form solution for which a combinatorial interpretation exists in terms of an infinite, labeled tree. For $s=0$, we show that this solution sequence has a closed form as the sum of ceiling functions $C(n) = \sum_{i=0}^{j-1}\cln{n-i}{2j}$. Further, given appropriate initial conditions, we derive necessary and sufficient conditions on the parameters $s_1, a_1, s_2$ and $a_2$ so that $C(n)$ solves the nested recursion $R(n) = R(n - s_1 -R(n-a_1)) + R(n-s_2-R(n-a_2))$.
\end{abstract}

\section{Introduction} \label{sec:Intro}
This paper investigates the occurrence of sums of ceiling functions as solutions to nested recursions
of the form
\begin{align}
R(n) = R(n - s_1 -R(n-a_1)) + R(n-s_2-R(n-a_2))
\label{Hn}
\end{align}
with $s_i,a_i$ integers, $a_i > 0$, and specified initial conditions. We adopt the terminology and notation from \cite{ConollyLike}, and write the above recursion as $\SEQ{s_1}{a_1}{s_2}{a_2}$.

The convergence of several recent discoveries has motivated our interest in such solutions. In \cite{BLT} the authors prove that the ceiling function $\cln{n}{2}$ solves the nested recursion $\SEQ{0}{1}{2}{3}$, with initial conditions 1,1,2.  In \cite{ConollyLike}, we vastly generalize this result by deriving necessary and sufficient conditions for the parameters $s_i,a_i$ so that $\cln{n}{2}$ solves the nested recursion $\SEQ{s_1}{a_1}{s_2}{a_2}$ with appropriate initial conditions. \footnote{It is also shown in \cite{ConollyLike} that for every $p \geq 1$, the ceiling function $\cln{n}{2p}$ solves an infinite family of order $p$ nested recursions. In this paper we restrict our attention to the recursion ($\ref{Hn}$), which has order 1.}

In a separate but related direction, in \cite{Rpaper} we solved a natural generalization of the recursion $\SEQ{0}{1}{2}{3}$, namely, $\SEQ{s}{j}{s+2j}{3j}$, with $s,j$ integers and $j$ positive. In so doing, we identified a closed form for the solution sequence that included a nesting of ceiling functions, albeit in a complicated way.

Finally, in \cite{CeilFunSol} we focused once again on the occurrence of certain ceiling function solutions, this time to nested recursions that naturally generalize the nested recursion ($\ref{Hn}$). In this case, for each $q>1$, we derived necessary and sufficient conditions on the parameters of the recurrence so that its solution has the closed form $\cln{n}{q}$ (given appropriate initial conditions).

Inspired by the repeated derivation of classification schemes for ceiling function solutions, we reexamine here the solutions to the family of recursions $\SEQ{s}{j}{s+2j}{3j}$ in \cite{Rpaper} from a ceiling function perspective. We prove that certain of these solutions have the closed form of a sum of ceiling functions. For these solutions, we discover a classification theorem analogous to the result in \cite{ConollyLike}; loosely speaking, we determine all of the possible nested recursions of this ``general" form which share the same solution sequence as $\SEQ{0}{j}{2j}{3j}$.

In the body of this paper we proceed as follows. In Section $\ref{sec:Exp}$, we examine the periodicity properties of sums of ceiling functions. In particular, we derive several useful properties of the ceiling function sum $C(n)$ defined by
\begin{align}
\label{eqn:solution}
C(n) = \sum_{i=0}^{j-1}\cln{n-i}{2j}.
\end{align}
We explain our interest in $C(n)$ in Section $\ref{sec:Main}$, where we prove that $C(n)$ solves $\SEQ{0}{j}{2j}{3j}$. For $s \neq 0$, we show in Section $\ref{sec:Exp}$ that we cannot write the solution to $\SEQ{s}{j}{s+2j}{3j}$ as a sum of ceiling functions. Thus, we must limit our findings to the $s=0$ case.

In Section $\ref{sec:Main}$ we apply the results of Section $\ref{sec:Exp}$ to derive a classification theorem for all the recursions $\SEQ{s_1}{a_1}{s_2}{a_2}$ that have the solution $C(n)$ with appropriate initial conditions.
In other words, we determine completely all the parameters $s_i,a_i$ for which $(\ref{eqn:solution})$ solves the recursion $\SEQ{s_1}{a_1}{s_2}{a_2}$. As a byproduct of this work, it follows that $C(n)$ solves $\SEQ{0}{j}{2j}{3j}$. In Section $\ref{sec:Conc}$ we conclude with some thoughts about potential further work in this general area.

\section{Periodicity and Sums of Ceiling Functions} \label{sec:Exp}

In this section, we examine some key properties of sequences that arise as sums of ceiling functions in general, and of the sum $C(n)$ in particular. We begin by determining which sequences have a closed form as a sum of ceiling functions.

\begin{theorem}
\label{thm:Periodic}
Let $\{a_n\}$ be an integer sequence. We can find a closed form $a_n = c+\sum_{i=1}^k \lceil{q_in+r_i}\rceil$ with $q_i,r_i$ rational if and only if the difference sequence $d_n = a_{n+1} - a_n$ is periodic.
\end{theorem}
\begin{proof}
First, suppose $a_n = c+\sum_{i=1}^k \lceil{q_in+r_i}\rceil$. The difference sequence of $a_n$ is the sum of the difference sequences of $\lceil q_in+r_i \rceil$, each of which is periodic (with period the denominator of $q_i$), and their sum is thus periodic (with period a divisor of the lowest common denominator of the $q_i, i=1,...,k$).

Now suppose $d_n$ is periodic with period $p$. Then consider $b_n = a_1+\sum_{i=1}^pd_i\cln{n-i}{p}$. First, note that $b_1=a_1$. Next, observe that $\cln{n-i}{p}$ has a difference sequence consisting of $0$ for $n \not \equiv i$ (mod p), $1$ for $n \equiv i$ (mod p). Thus, $d_i\cln{n-i}{p}$ has a difference sequence of 0 for $n \not \equiv i$ (mod p), $d_i$ for $n \equiv i$ (mod p), and so the difference sequence of $b_n$ is just $b_{n+1}-b_n = d_i$ where $0<i\leq p$ and $i \equiv n$ (mod p). But since $d_n$ is periodic with period $p$, this means the difference sequence of $b_n$ is $d_n$, and so $a_n$ and $b_n$ have the same first element and the same difference sequence and are therefore equal.
\end{proof}

In \cite{Rpaper}, we derived a closed form for the solution to the recurrence $\SEQ{s}{j}{s+2j}{3j}$. This formula clearly shows that the solution has a periodic difference sequence if and only if $s=0$. Thus, in the remainder of this paper, we restrict ourselves to examining the solution to the $s=0$ case.

An interesting result related to Theorem $\ref{thm:Periodic}$ follows:

\begin{theorem}
If the solution sequence to a nested recursion has a periodic difference sequence, then this same sequence also solves a non-nested recursion.
\end{theorem}
\begin{proof}
By Theorem $\ref{thm:Periodic}$, such a solution sequence has a closed form $c+\sum_{i=1}^k \lceil{q_in+r_i}\rceil$; without loss of generality, rewrite the $q_i$ to have a common denominator so that $q_i = b_i/q$ with $b_i,q$ integers. Then the same solution sequence solves the recursion $A(n) = A(n-q)+\sum_{i=1}^kb_i$.
\end{proof}

Based on the structure of the solution to $\SEQ{0}{j}{2j}{3j}$ found in \cite{Rpaper}, we can observe that the sequence ($\ref{eqn:solution}$)
solves $\SEQ{0}{j}{2j}{3j}$. We will prove this fact in the next section, but first we prove here two lemmas which simplify computations involving $C(n)$.

\begin{lemma}
For any $n$,$d \in \mathbb{Z}$, $C(n+2jd) = C(n)+jd$.
\label{lm:C}
\end{lemma}

\begin{proof}
We have
$C(n+2jd) = \sum_{i=0}^{j-1}\cln{n+2jd-i}{2j} =
 \sum_{i=0}^{j-1}(\cln{n-i}{2j}+d)= \sum_{i=0}^{j-1}\cln{n-i}{2j}+jd=C(n)+jd$ and this completes the proof.
\end{proof}

This lemma shows that if we know the values of $C(n)$ for $2j$ consecutive values of $n$, then we can easily compute the rest of the sequence. In the next lemma, we find the values of $C(n)$ for $0 \leq n \leq 2j-1$.

\begin{lemma}
$C(n)=n$ for $1 \leq n \leq j-1$ and $C(n)=j$ for $j \leq n \leq 2j$.
\label{lm:freq}
\end{lemma}

\begin{proof}
Let $n\in \{1,2,...,j-1\}$ and $i \in \{0,1,2,...,j-1\}$.  Then $-j+2 \leq n-i \leq j-1$. Therefore, we have that $\cln{n-i}{2j}$ is $1$ when $n>i$ and it vanishes otherwise. Hence, we have $C(n) = \sum_{i=0}^{j-1}\cln{n-i}{2j} = \sum_{i=0}^{n-1}1 + \sum_{i=n}^{j-1}0 = n$. Now let $n\in \{j,j+1,...,2j\}$ and $i \in \{0,1,2,...,j-1\}$. Then $1 \leq n-i \leq 2j$, which implies that $\cln{n-i}{2j}=1$. Hence,

\begin{align*}
C(n) = \sum_{i=0}^{j-1}\cln{n-i}{2j} = \sum_{i=0}^{j-1}1 = j
\end{align*}

as required.

\end{proof}


\section{Finding All Recursions $\SEQ{s_1}{a_1}{s_2}{a_2}$ Solved By $C(n)$} \label{sec:Main}

In this section we determine all of the recursions $R(n) = \SEQ{s_1}{a_1}{s_2}{a_2}$ solved by $C(n)$ when given appropriate initial conditions. In so doing we make use of the idea of ``formal satisfaction." We say an infinite sequence formally satisfies a recursion if the recursive formula is well-defined and true on that sequence for all integers. By contrast, a sequence is generated as the unique solution to a recursion and a set of $c$ specific initial conditions if for all $n>c$, the recursion allows us to calculate the value of the solution sequence at $n$ by referencing only terms with indices less than $n$.

A simple example will clarify this distinction. The recursion $S(n) = S(S(n+1))$ is formally satisfied by the sequence $S(n)= 1$ for all $n$. However, for any positive integer $c$, if we are given the initial conditions $S(1) = S(2) = \ldots = S(c) = 1$, we cannot determine the value of $S(c+1)$ because the recursion requires we know the value of $S(c+2)$.

Thus, in general, formal satisfaction does not imply generation as an infinite solution sequence. But for the particular case we are dealing with, namely, the recursion $R(n) = \SEQ{s_1}{a_1}{s_2}{a_2}$ and the sequence $C(n)$, formal satisfaction does imply generation as an infinite solution sequence. To see why, note that $C(n)$ asymptotically approaches $n/2$, so $n-s_1-C(n-a_{1})$ and $n-s_2-C(n-a_{2})$ also asymptotically approach $n/2$. Furthermore, as long as $a_i>0$, for large enough $n$ the recursion for $R(n)$ refers only to prior positive terms and can thus be generated given sufficiently many appropriate initial conditions.

Although it might at first seem like an additional complication, the idea of formal satisfaction simplifies many of our proofs - in many cases, we can most easily prove that a recursion generates $C(n)$ as an infinite solution sequence by proving that $C(n)$ formally satisfies the recursion. Thus, we now proceed with a theorem classifying all recursions $R(n) = \SEQ{s_1}{a_1}{s_2}{a_2}$ that $C(n)$ formally satisfies.

\begin{theorem}
$C(n) = \sum_{i=0}^{j-1}\cln{n-i}{2j}$ formally satisfies the nested recursion $\SEQ{s_1}{a_1}{s_2}{a_2}$ if and only if the following conditions hold:

\begin{align*}
s_1,s_2 \equiv 0 \bmod{j} \tag{i}\\
a_1,a_2 \equiv j \bmod{2j} \tag{ii}\\
2(s_1+s_2) = a_1 + a_2 \tag{iii}
\end{align*}
\label{thm:Main}
\end{theorem}

Notice that for $j=1$, the conditions on the parameters reduce to the characterization of all nested recursions $\SEQ{s_1}{a_1}{s_2}{a_2}$ formally satisfied by the ceiling function $\cln{n}{2}$ (derived in \cite{ConollyLike}).

To show that the conditions listed above suffice for $C(n)$ to formally satisfy $\SEQ{s_1}{a_1}{s_2}{a_2}$, we adapt the proof technique used in \cite{CeilFunSol} to prove an analogous result classifying all nested recursions formally satisfied by $\cln{n}{q}$. The basic elements of our approach follow: first, we establish a natural equivalence relation on the set of all recursions of the form $\SEQ{s_1}{a_1}{s_2}{a_2}$. Next, we show that if the sequence ($\ref{eqn:solution}$) formally satisfies one element of an equivalence class, then it satisfies every element of that equivalence class. Then we prove that every equivalence class has a representative in the set $\mathbb{Z}\times S \times S \times S$, where $S = \{0,1,2,...,2j-1\}$. Finally, we demonstrate that if $C(n)$ formally satisfies $\SEQ{s_1}{a_1}{s_2}{a_2}$ for $4j$ consecutive values of $n$, then it does so for all $n$. Putting all these facts together, we conclude by directly verifying that when the conditions listed above hold, then $C(n)$ satisfies $\SEQ{s_1}{a_1}{s_2}{a_2}$ for $0 \leq n \leq 4j-1$.

We now proceed with a series of five lemmas.
To establish a natural equivalence relation on the set of all recursions of the form $\SEQ{s_1}{a_1}{s_2}{a_2}$, we treat $\SEQ{s_1}{a_1}{s_2}{a_2}$ as a vector in $\mathbb{Z}^4$, denoted by $y$. Then we define the equivalence relation $\sim$ on the set of vectors $y$:
\begin{align*}
\SEQ{s_1}{a_1}{s_2}{a_2} \sim \SEQ{s_1 + cj}{a_1 + 2cj}{s_2}{a_2} \tag{a}\\
\SEQ{s_1}{a_1}{s_2}{a_2} \sim \SEQ{s_1}{a_1}{s_2 + dj}{a_2 + 2dj} \tag{b}\\
\SEQ{s_1}{a_1}{s_2}{a_2} \sim \SEQ{s_1 - 2ej}{a_1}{s_2 + 2ej}{a_2} \tag{c}
\end{align*}
where $c$, $d$, $e \in \mathbb{Z}$. Our first lemma shows that if any element of an equivalence class satisfies conditions (i)-(iii) then every element of that equivalence class satisfies those conditions.

\begin{lemma}
Let $y$ satisfy (i)-(iii). If $y \sim y'$, then $ y'$ satisfies (i)-(iii).
\label{lm:CondPreserve}
\end{lemma}

\begin{proof}
It suffices to verify the statement of the theorem for relations (a), (b)
and (c) separately. In the following, let $y = \SEQ{s_1}{a_1}{s_2}{a_2}$ and $ y' = \SEQ{s_1'}{a_1'}{s_2'}{a_2'}$.

We first check that equivalence under (a) preserves (i)-(iii). Assume $ y' = \SEQ{s'_1}{a'_1}{s'_2}{a'_2} = \SEQ{s_1 + cj}{a_1 + 2cj}{s_2}{a_2}$ for some $c \in \mathbb{Z}$. Then since
$s_1 \equiv 0 \bmod{j}$ and
$a_1 \equiv j \bmod{2j}$ by assumption, we have that
$
s_1 + cj \equiv 0 + cj \equiv 0 \bmod{j}$ and
$a_1 + 2cj \equiv j + 2cj \equiv j \bmod{2j}.$ Furthermore, since $2(s_1 + s_2) = a_1 + a_2$, we have $2(s_1 + cj + s_2) = 2(s_1 + s_2) + 2cj = a_1 + a_2 + 2cj = (a_1 + 2cj) + a_2$. It follows that $s'_1$, $a'_1$, $s'_2$ and $a'_2$ satisfy (i)-(iii).

The argument for (b) is identical and therefore omitted.

Finally, we verify that equivalence under (c) preserves (i)-(iii). Let $ y' = \SEQ{s'_1}{a'_1}{s'_2}{a'_2} = \SEQ{s_1 - 2ej}{a_1}{s_2 + 2ej}{a_2}$ for some $e \in \mathbb{Z}$. By assumption  $s_1 \equiv 0 \bmod{j}$ and $s_2 \equiv 0 \bmod{j}$, so we have $
s_1 - 2ej \equiv 0 - 2ej \equiv 0 \bmod{j},
s_2 + 2ej \equiv 0 + 2ej \equiv 0 \bmod{j}$, and $2(s_1 -2ej + s_2 + 2ej) = 2(s_1 + s_2) = a_1 + a_2.$

This shows that $s'_1$, $a'_1$, $s'_2$ and $a'_2$ satisfy (i)-(iii) as required, thereby completing the proof.
\end{proof}

Now we show that if the sequence ($\ref{eqn:solution}$) formally satisfies one element of an equivalence class, then it satisfies every element of that equivalence class.
Define the difference function  $h(n, y) = C(n-s_1-C(n-a_1))+C(n-s_2-C(n-a_2))-C(n)$. Observe that for a fixed $y$, the sequence ($\ref{eqn:solution})$ formally satisfies $y$ if and only if $h(n,y) = 0$ for all $n \in \mathbb{Z}$. We have the following lemma.

\begin{lemma}
If $y \sim y'$ then $ h(n, y)= h(n, y').$
\label{lm:hInvariant}
\end{lemma}

\begin{proof}
As before, it suffices to prove this lemma separately for relations (a), (b) and (c).

First, we verify (a) preserves $h$. Let
$y =  \SEQ{s_1}{a_1}{s_2}{a_2}$ and
$y' = \SEQ{s_1 + cj}{a_1 + 2cj}{s_2}{a_2}$ for some $c \in \mathbb{Z}$.
Note that by Lemma \ref{lm:C} we have $
C(n-s_1-cj-C(n-a_1-2cj)) =
C(n-s_1-cj-(C(n-a_1)-cj)) =
C(n-s_1-cj+cj-(C(n-a_1))=
C(n-s_1-C(n-a_1)).
$ Hence, $h(n, y') = C(n-s_1-cj-C(n-a_1-2cj))+C(n-s_2-C(n-a_2))-C(n) =
 C(n-s_1-C(n-a_1))+C(n-s_2-C(n-a_2))-C(n) =
h(n, y)
$

The same argument applies to confirm that (b) preserves $h$; we omit the details.

We check that (c) preserves $h$. Assume
$y =  \SEQ{s_1}{a_1}{s_2}{a_2}$ and
$y' = \SEQ{s_1-2ej}{a_1}{s_2+2ej}{a_2}$ for some $e \in \mathbb{Z}$.

Applying Lemma \ref{lm:C}, we get $
h(n, y') = C(n-s_1+2ej-C(n-a_1))+C(n-s_2-2ej-C(n-a_2))-C(n) =
 C(n-s_1-C(n-a_1))+ej+C(n-s_2-C(n-a_2))-ej-C(n) =
h(n, y)$ and this completes the proof.
\end{proof}

Now we show that every equivalence class has a representative in the set $\mathbb{Z}\times S \times S \times S$, where $S = \{0,1,2,...,2j-1\}$. As we will see later, this result, together with the following lemma and conditions (i)-(iii), will reduce the verification of formal satisfaction to a finite number of confirmatory calculations.

\begin{lemma}
For every $y \in \mathbb{Z}^4$, there exists $ y' \in \mathbb{Z}\times S \times S \times S$ such that $ y \sim y'$.
\label{lm:repBouned}
\end{lemma}

\begin{proof}
Consider an arbitrary $ y =  \SEQ{s_1}{a_1}{s_2}{a_2} \in \mathbb{Z}^4$.
Note that by the division algorithm $a_2 = 2jq + a'_2$ for some $q \in \mathbb{Z}$ and $a'_2 \in S$. Then, we apply (b) to obtain $y =  \SEQ{s_1}{a_1}{s_2}{2jq+a'_2} \sim \SEQ{s_1}{a_1}{s_2-jq}{2jq+a'_2-2jq} = \SEQ{s_1}{a_1}{s'_2}{a'_2}$.

As before, $s'_2 = 2jc+s''_2$ for some $c \in \mathbb{Z}$ and $s''_2 \in S$.
Now, we apply (c) to get $\SEQ{s_1}{a_1}{s'_2}{a'_2} = \SEQ{s_1}{a_1}{2jc+s''_2}{a'_2} \sim \SEQ{s_1+2jc}{a_1}{2jc+s''_2-2jc}{a'_2} = \SEQ{s'_1}{a_1}{s''_2}{a'_2}$.

Finally, we use the fact that $a_1 = 2jd +a'_1$, where $d \in \mathbb{Z}$ and $a'_1 \in S$, and (a) to get $\SEQ{s'_1}{a_1}{s''_2}{a'_2} = \SEQ{s'_1}{2jd + a'_1}{s''_2}{a'_2} \sim \SEQ{s'_1-jd}{2jd + a'_1-2jd}{s''_2}{a'_2} = \SEQ{s''_1}{a'_1}{s''_2}{a'_2}$.

Therefore, $ y \sim y' = \SEQ{s''_1}{a'_1}{s''_2}{a'_2}$, where $a'_1$,$s''_2$,$a'_2 \in S$ and $s''_1 \in \mathbb{Z}$.

\end{proof}

Currently, to check that $C(n)$ formally satisfies some recursion corresponding to $y$, we have to check that $h(n, y) = 0$ for each $n \in \mathbb{Z}$. Our next lemma remedies this situation by reducing to finitely many $n$.

\begin{lemma}
For a fixed $y$ and any integer $d$, $h(n, y) = h(n + 4jd, y)$.
\label{lm:nBounded}
\end{lemma}

\begin{proof}
Fix $y =  \SEQ{s_1}{a_1}{s_2}{a_2}$ and an integer $d$.
Applying  Lemma \ref{lm:C}, we have $
C(n+4jd-s_1-C(n+4jd-a_1)) =
C(n+4jd-s_1-(C(n-a_1)+2jd)) =
C(n+2jd-s_1-C(n-a_1)) =
C(n-s_1-C(n-a_1)) + jd
$ and similarly $C(n+4jd-s_2-C(n+4jd-a_2)) = C(n-s_2-C(n-a_2)) + jd$. Then,
we calculate $
h(n+4jd, y) = C(n+4jd-s_1-C(n+4jd-a_1)) + C(n+4jd-s_2-C(n+4jd-a_2)) - C(n+4jd) =
 C(n-s_1-C(n-a_1)) + jd + C(n-s_2-C(n-a_2)) + jd - C(n) -2jd =
 C(n-s_1-C(n-a_1)) + C(n-s_2-C(n-a_2)) - C(n) =
 h(n, y)$ and this completes the proof.
\end{proof}

Lemma \ref{lm:nBounded} has a key consequence: to confirm that $h(n, y) = 0$ for all $n \in \mathbb{Z}$, it suffices to verify that $h(n, y) = 0$ for $0 \leq n \leq 4j-1$.

The above results provide all the necessary tools to show that conditions (i)-(iii) suffice for the sequence $(\ref{eqn:solution})$ to formally satisfy $ y$.

\begin{lemma}
Let $ y$ satisfy conditions (i)-(iii). Then the sequence $(\ref{eqn:solution})$ formally satisfies $y$.
\label{lm:sufficient}
\end{lemma}

\begin{proof}
Observe that by Lemma \ref{lm:repBouned} there exists $y' = \SEQ{s'_1}{a'_1}{s'_2}{a'_2}$ such that $ y' \sim y$ and $ y' \in \mathbb{Z} \times S \times S \times S$. Notice that, by Lemma \ref{lm:CondPreserve}, $s_1'$,$a_1'$,$s_2'$ and $a_2'$ also satisfy conditions (i)-(iii). Then, $a'_1 = a'_2 = j$ and $s'_2$ is 0 or $j$. If $s'_2$ is 0, then from condition (iii) it follows that $s'_1 = j$. Alternatively, if $s'_2 = j$ then condition (iii) implies that $s'_1 = 0$. Either way (switching the order of the summands if needed), $ y'$ corresponds to the recursion $R(n) = R(n - R(n-j)) + R(n-j-R(n-j))$. Therefore, without loss of generality we assume that  $y' = \SEQ{0}{j}{j}{j}$.

By Lemma \ref{lm:hInvariant}, to show that $h(n, y) = 0$ for all integers $n$, it suffices to show that $h(n, y') = 0$ for all integers $n$. Furthermore, Lemma \ref{lm:nBounded} shows that we only need to show that $h(n, y') = 0$ for $0 \leq n \leq 4j-1$. By Lemmas \ref{lm:C} and $\ref{lm:freq}$ we have

\begin{align*}
C(n) =
\begin{cases}
0,   & -j \leq n \leq -1 \\
n,   & 0 \leq n \leq j-1 \\
j,   & j \leq n \leq 2j-1 \\
n-j, & 2j \leq n \leq 3j-1 \\
2j,  & 3j \leq n \leq 4j-1
\end{cases}
\end{align*}

We need to consider several cases.
First, suppose that $0 \leq n \leq j-1$. Then $C(n-j)=0$ since $-j \leq n-j \leq -1$. Therefore, $h(n,y') = C(n-C(n-j)) + C(n-j-C(n-j)) - C(n) = C(n) + C(n-j) - C(n) = 0$, as required.

Next consider the case when $j \leq n \leq 2j-1$. Hence, $C(n-j)=n-j$ as $0 \leq n-j \leq j-1$. Thus, $h(n, y') = C(n-C(n-j)) + C(n-j-C(n-j)) - C(n) =
            C(j) + C(0) - C(n) =
            j + 0 - j = 0$.

Now let $2j \leq n \leq 3j-1$. In this case, $C(n-j)=j$ and $C(n-2j)=n-2j$ since $j \leq n-j \leq 2j-1$ and $0 \leq n-2j \leq j-1$. Hence, $
h(n, y') = C(n-C(n-j)) + C(n-j-C(n-j)) - C(n) =
            C(n-j) + C(n-2j) - C(n) =
            j + n-2j - n+j = 0$.

Finally, suppose that $3j \leq n \leq 4j-1$. Then $C(n-j)=n-2j$ since $2j \leq n-j \leq 3j-1$. Then $h(n, y') = C(n-C(n-j)) + C(n-j-C(n-j)) - C(n)
           = C(2j) + C(j) - C(n) =
            j + j - 2j = 0 $ as required. We conclude that $h(n, y')=0$ for $0 \leq n \leq 4j-1$ and this completes the proof of the lemma.

\end{proof}

We conclude the proof of Theorem \ref{thm:Main} by showing the necessity of conditions (i)-(iii) for $(\ref{eqn:solution}$) to formally satisfy $ y$.

\begin{lemma}
Let the sequence $(\ref{eqn:solution}$) formally satisfy $ y$. Then $ y$ satisfies conditions (i)-(iii).
\label{lm:Necess}
\end{lemma}

\begin{proof}
By Lemma \ref{lm:repBouned}, $ y \sim  y'$ for some $ y' \in \mathbb{Z} \times S \times S \times S$, where $S = \{0,1,2,...,2j-1\}$. Furthermore, by Lemma \ref{lm:CondPreserve}, $ y$ satisfies (i)-(iii) if and only if $ y'$ does. Therefore, without loss of generality, we may assume that $ y \in \mathbb{Z} \times S \times S \times S$. We now prove that $ y = \SEQ{0}{j}{j}{j}$; since $\SEQ{0}{j}{j}{j}$ clearly satisfies conditions (i)-(iii), this will complete the proof of this lemma.

By assumption, $(\ref{eqn:solution}$) formally satisfies $ y$, so $C(n) = C(n-s_1-C(n-a_1)) + C(n-s_2-C(n-a_2))$ for all $n$. By the Euclidean division algorithm, $s_1 = 2jk + s$ for some $k \in \mathbb{Z}$ and $s \in S$. Therefore, by Lemma \ref{lm:C}, $C(n) = C(n-s-C(n-a_1)) + C(n-s_2-C(n-a_2))-jk$. Further, as in Lemma $\ref{lm:sufficient}$, we will use the following values of $C(n)$, which follow from Lemmas \ref{lm:C} and \ref{lm:freq}:

\begin{align*}
C(n) =
\begin{cases}
n+j   &-2j \leq n \leq -j-1 \\
0,    &-j \leq n \leq 0 \\
n,    & 1 \leq n \leq j-1 \\
j,    & j \leq n \leq 2j
\end{cases}
\end{align*}

In particular, observe that for $n$ satisfying $-2j \leq n < 2j$, $C(n)$ is constant precisely on the intervals $[-j,0]$ and $[j,2j]$ only; that is, if $C(n) = C(n+1)$ for some $n$ with $-2j \leq n < 2j$, then $n$ must satisfy $-j \leq n < 0$ or $j \leq n < 2j$. We repeatedly use this observation.

First, we demonstrate that one of the summands $C(n-s-C(n-a_1))$ or $C(n-s_2-C(n-a_2))$ is in fact $C(n-C(n-j))$; that is, we show that either $s=0,a_1=j$ or $s_2=0,a_2=j$.
Note that the sequences defined by $C(n-s-C(n-a_1))$ and $C(n-s_2-C(n-a_2))$ are both slow, that is, their forward differences equal either $0$ or $1$. This follows directly from the fact that $C(n)$ itself is slow, which in turn follows from Lemmas \ref{lm:freq} and \ref{lm:C}. Our main use of this fact is to note that if $C(n)$ stays constant, both of its summands must stay constant, and if $C(n)$ increases by 1, then exactly one of its summands must increase by $1$.

Since $C(0) = 0$ and $C(1) = 1$, one of the summands must have increased by $1$;  without loss of generality we may assume it was $C(n-s-C(n-a_1))$, interchanging the summands if needed. Hence, $C(1-s-C(1-a_1)) = 1+C(-s-C(-a_1))$. Since $C(n)$ is slow, either $C(-a_1)=C(1-a_1)$, or $C(-a_1)+1=C(1-a_1)$. In the latter case, we would have $1-s-C(1-a_1)=-s-C(-a_1)$, contradicting $C(1-s-C(1-a_1)) = 1+C(-s-C(-a_1))$. Thus, $C(-a_1)=C(1-a_1)$ and so $C(1-s-C(-a_1))=1+C(-s-C(-a_1))$.

Since $C(-a_1) = C(1-a_1)$ and $a_1 \in S$, it must be the case that $C(-a_1) = 0$ and $0<a_1\leq j$ (see the listing of values of $C(n)$ above). Furthermore, since $C(-a_1) = 0$, it follows (by substituting into the last equation in the previous paragraph) that $C(1-s) = 1+C(-s)$. This implies that either $s=0$ or $j<s<2j$.

Summarizing the above results, we have the following restrictions: $0<a_1\leq j$, and either $s=0$ or $j<s<2j$.

Next, we show that $a_1 = j$. If not, then $0<a_1<j$. By the list of values of $C(n)$ above, $C(j+a_1)=C(j+a_1+1)= \ldots = C(2j) = j$. Therefore, since $C(n)$ is constant as $n$ ranges from $j+a_1$ to $2j$, both of its summands must be constant on the same range. In particular, $C(n-s-C(n-a_1))$ must stay constant as $n$ ranges from $j+a_1$ to $2j$. Thus $C(j+a_1-s-C(j+a_1-a_1)) = \ldots = C(2j-s-C(2j-a_1))$. Since $C(j+a_1-a_1) = \ldots = C(2j-a_1) = j$, we have that $C(j+a_1-s-j) = ... = C(2j-s-j)$, or, simplifying, $C(a_1-s) = C(a_1+1-s) = ... = C(j-s)$. This implies that $s \neq 0$ since otherwise $C(j) = C(a_1)$, where $a_1 < j$. Hence, we must have $j<s<2j$.

But observe that $C(n)$ also remains constant and equal to 0 as $n$ ranges from $-j$ to $0$. Applying the same argument as above with all the terms shifted back by 2j, we can conclude that $C(-j+a_1-s) = C(-j+1+a_1-s) = ... = C(-s)$. But this is impossible, since for $j<s<2j$ we have $C(-s) = -s+j$ and $C(-s-1) = -s-1+j$ (see the list of values of $C(n)$ above). Therefore, we must have $a_1 = j$.

Summarizing our current situation, we have $a_1=j$ and either $s=0$ or $j<s<2j$.

Now we show that $s=0$. If not, then $j<s<2j$. As an immediate consequence, $C(j-1-s-C(j-1-a_1)) = C(j-1-s)=0=C(j-s)=C(j-s-C(j-a_1))$ where the first and last equalities come from substituting $a_1=j$, and the middle two equalities hold because $j-s$ and $j-1-s$ lie between $-j$ and $0$, and thus $C(j-s)=C(j-1-s)=0$. Since $C(j-1) + 1 = C(j)$ and $C(j-1-s-C(j-1-a_1)) = C(j-s-C(j-a_1))$, the second summand of $C(n)$ must be the one to increase as $n$ changes from $j-1$ to $j$, so $C(j-1-s_2-C(j-1-a_2)) + 1 = C(j-s_2-C(j-a_2))$. Therefore, we now turn our attention to the term $C(n-s_2-C(n-a_2))$.

We have that $C(j-1-s_2-C(j-1-a_2)) + 1 = C(j-s_2-C(j-a_2))$. Thus, the arguments $j-1-s_2-C(j-1-a_2)$ and $j-s_2-C(j-a_2)$ must be different, which requires $C(j-1-a_2)=C(j-a_2)$. Since $0\leq a_ 2 < 2j$, we have that $-j < j - a_2 \leq j$, which together with $C(j-1-a_2) = C(j-a_2)$ implies that $-j < j-a_2 \leq 0$ (see the list of values for $C(n)$ above). Thus, $C(j-1-a_2)=C(j-a_2)=0$. So we can conclude $C(j-a_2) = C(j-1-a_2) = 0$ and $j \leq a_2 < 2j$. Furthermore, this implies that $C(j-s_2) = C(j-1-s_2)+1$, so $0 \leq s_2 < j$. Consider first the case that $a_2 \neq j$, so that $j<a_2<2j$. Since $C(n)$ is constant as $n$ ranges from $j$ to $a_2$, its second summand must be constant on this range so $C(j-s_2-C(j-a_2)) = ... = C(a_2-s_2-C(a_2-a_2))$. However, since $C(j-a_2) = ... = C(a_2-a_2) = 0$, this tells us that $C(j-s_2) = ... = C(a_2-s_2)$, which is only possible if $s_2 = 0$. Next, since $C(n)$ is also constant as $n$ ranges from $-j$ to $-2j+a_2$, we can apply the preceding argument shifted back by $2j$ terms to conclude that $C(-j+j) = ... = C(-2j+a_2+j)$, which we rewrite as $C(0) = ... = C(a_2-j)$. This is impossible since $C(a_2-j) \neq C(0)$. Thus we conclude that the case $j<a_2<2j$ cannot occur, and we now let $a_2=j$.

Recall that we are still working under the assumption that $j<s<2j$, and we have shown that $a_1=j$ and now $a_2=j$, and that $0 \leq s_2 < j$. Note that $0=C(0) = C(-s-C(-j)) + C(-s_2 - C(-j)) -jk$. Since $C(-j) = 0$, this reduces to $0 = C(-s) + C(-s_2)-jk$. However, $0 \leq s_2 < j$ and consulting the list of values for $C(n)$, we see $C(-s_2) = 0$. This implies $C(-s) = jk$. Since $j<s<2j$ we have $0 > C(-s) > -j$, giving a contradiction. Therefore, we conclude that $s=0$.

We now have proved that $s = 0$ and $a_1=j$, so $C(n) = C(n-C(n-j)) + C(n-s_2-C(n-a_2)) -jk$. We still don't know anything about $s_2$ or $a_2$.

A key property of $C(n-C(n-j))$  is that it remains constant as $n$ ranges from $2j$ to $3j$. Indeed, by the listing of values of $C(n)$ above, if $2j \leq n \leq 3j$, then $C(n-j)=j$, so $n-C(n-j)=n-j$, so $C(n-C(n-j))=j$, again by the listing of values of $C(n)$.

By Lemmas \ref{lm:freq} and \ref{lm:C}, $C(n+1)=C(n)+1$ for $2j \leq n < 3j$. But the first summand $C(n-C(n-j))$ remains constant on this range, so it follows that the second summand $C(n-s_2-C(n-a_2))$ must be increasing on this range. That is, for $2j \leq n < 3j$, we have $1+C(n-s_2-C(n-a_2))=C(n+1-s_2-C(n+1-a_2))$. This in turn implies that for $n$ in this range, $n-s_2-C(n-a_2) \neq n+1-s_2-C(n+1-a_2)$, which simplifies to $C(n+1-a_2) \neq 1+C(n-a_2)$. Since $C(n)$ increases only by 0 or 1, it must be the case that $C(n+1-a_2)=C(n-a_2)$ for all $n$ in the range $2j \leq n < 3j$. By consulting the list of values of $C(n)$ above, we see that the sequence $2j-a_2,2j+1-a_2,\ldots,3j-a_2$ must begin with $(2z+1)j$ for some integer $z$ (so that $C(n)$ is constant on the next $j$ values). But $a_2$ lies in the range $0 \leq a_2 < 2j$, so the only possibility is $2j-a_2=j$, or $a_2=j$.

In the previous paragraph, we showed that for $2j \leq n < 3j$, we have $1+C(n-s_2-C(n-a_2))=C(n+1-s_2-C(n+1-a_2))$. But $a_2=j$,  so $j \leq n-a_2<2j$ and therefore by the list of values of $C(n)$, $C(n-a_2)=C(n+1-a_2)=j$. Thus, $1+C(n-s_2-j)=C(n+1-s_2-j)$ for all $n$ in the range $2j \leq n < 3j$. Consulting the list of values for $C(n)$ above, we see that $2j-s_2-j$ must be $2zj$ for some integer $z$, to ensure that $C(n)$ increases on the next $n$ values. Then since $0 \leq s_2 < 2j$, $2j-s_2-j=0$ is the only possibility so $s_2=j$.

Thus, we have that $C(n) = C(n-C(n-j))+C(n-j-C(n-j))-kj$. By substituting $n=j$ we immediately deduce that $k=0$, so $C(n) = C(n-C(n-j))+C(n-j-C(n-j))$ as desired. This completes the proof.
\end{proof}

Together, Lemmas \ref{lm:sufficient} and \ref{lm:Necess} prove Theorem \ref{thm:Main}. Further, observe that this also proves that $C(n)$ formally satisfies $R(n) = R(n-R(n-j)) + R(n-2j-R(n-3j))$ since $\SEQ{0}{j}{j}{j} \sim \SEQ{0}{j}{2j}{3j}$.

\section{Concluding Remarks} \label{sec:Conc}
A wide variety of nested recursions have solutions exhibiting either periodic or ``periodic-like" behaviour. In \cite{Golomb1990}, Golomb observed that if the one-term nested recursion $a_n = a_{n-a_{n-1}}$ has a solution, then it always eventually becomes periodic; in fact, this holds for any one-term homogeneous nested recursion. This paper, as well as \cite{Rpaper}, \cite{CeilFunSol}, and \cite{ConollyLike}, have exhibited large families of nested recursions with periodic difference sequences. In \cite{Golomb1990}, Golomb illustrated that with appropriate initial conditions Hofstadter's Q-recursion $Q(n) = Q(n-Q(n-1))+Q(n-Q(n-2))$ could be made to exhibit what he called ``quasi-periodic" behavior: more precisely, given initial conditions $Q(1) = 3, Q(2)=2$, and $Q(3)=1$, the resulting solution has $Q(3k+1)=3, Q(3k+2)=3k+2$ and $Q(3k)=3k-2$. Ruskey \cite{FibHof} has demonstrated similar behaviour involving the Q-recursion and the Fibonacci sequence.

All together, this suggests that ``periodic-like" behavior appears frequently in the solutions to nested recursions. Perhaps more such periodicity variants await discovery. Even further, perhaps some property, such as ``$a_n-a_{n-p}$ is periodic for some $p$", may unify the known examples and lead to a broader result about all such solution sequences.

\end{document}